\documentclass[a4paper,11pt]{article}
\usepackage[colorlinks,bookmarksopen,bookmarksnumbered,citecolor=blue, linkcolor=black, urlcolor=blue]{hyperref}
\usepackage{graphicx}
\usepackage{epstopdf}
\usepackage{indentfirst}
\usepackage{slashbox}
\usepackage[centertags]{amsmath}
\usepackage{amsfonts}
\usepackage{amssymb}
\usepackage{amsthm}
\usepackage{epsfig}
\usepackage{fancyhdr}
\usepackage{mathrsfs}
\usepackage{amsmath}
\usepackage{multirow}
\usepackage[misc]{ifsym}
\usepackage{booktabs}
\usepackage{stmaryrd}
\usepackage{algorithm}
\usepackage{algorithmic}
\usepackage{natbib}
\usepackage{enumerate}
\hypersetup{hypertex=true,
	colorlinks=true,
	linkcolor=blue,
	anchorcolor=blue,
	citecolor=blue,
	pdfstartview=Fit,
	breaklinks=true}

\allowdisplaybreaks[4]
\numberwithin{equation}{section}

\begin{document}
\newtheorem{theorem}{\bf Theorem\,}[section]
\newtheorem{lemma}{\bf Lemma\,}[section]
\newtheorem{remark}{\bf Remark\,}
\newtheorem{example}{\bf Example\,}
\newtheorem{proposition}{\bf Propositiona\,}[section]
\newtheorem{corollary}{\bf Corollary\,}[section]
\newtheorem{assumption}{\bf Assumption\,}
\newtheorem{definition}{\bf Definition\ }[section]

\theoremstyle{remark}
\newtheorem*{notation}{\bf Notation}

\renewcommand{\tablename}{Table}
\renewcommand\refname{\large References}
\date{}	
	\title{\textbf{Asymptotic  representations for Spearman's footrule correlation coefficient}}
	 \author{Liqi Xia$^1$, Li Guan$^2$ and Weimin Xu$^{3*}$
	 	 \bigskip\\
		1. School of Mathematics, Qilu Normal University, Jinan 250200, China \bigskip\\
      2. School of Mathematics, Statistics and Mechanics, Beijing University of Technology,\\ Beijing 100124, China \bigskip\\
      3. School of Data Science and Artificial Intelligence, Wenzhou University of Technology,\\ Wenzhou 325000, China \bigskip\\
		E-mail address: wzxuweimin@163.com (Weimin Xu)}
	\maketitle
	
	\begin{abstract}
In order to address the theoretical challenges arising from the dependence structure of ranks in Spearman's footrule correlation coefficient, we propose two asymptotic representations to approximate the distribution of this coefficient under the hypothesis of independence. The first representation simplifies the dependence structure by replacing empirical distribution functions with their population counterparts. The second representation leverages the H\'{a}jek projection technique to decompose the initial form into a sum of independent components, thereby rigorously justifying asymptotic normality. Simulation studies demonstrate the appropriateness of two proposed asymptotic representations, as well as their excellent approximation to the limiting normal distribution.	\end{abstract}

\section{Introduction}\label{section1}
Nonparametric measures of association play a pivotal role in statistical inference, particularly when data violate parametric assumptions or exhibit complex dependencies. Among these measures, Spearman's footrule rank correlation coefficient as a rank-based metric  (\cite{spearman1906footrule}), has garnered renewed interest due to its robustness and interpretability (\cite{bukovvsek2022exact,chen2023asymptotic,perez2023nonparametric}). Adding up the absolute differences between two sets of ranks, Spearman's footrule quantifies disarray between permutations, offering a natural alternative to Euclidean-based metrics like Spearman's rho. Furthermore, it also possesses an intuitive population. For continuous random variables $X$ and $Y$ with an underlying copula $C$, the population version of Spearman's footrule is defined as:
$$\varphi_C = 1 - 3 \int_{[0,1]^2} |u - v|  dC(u, v),
$$
where $u$ and $v$ represent the marginal distribution functions of $X$ and $Y$, respectively. Under independence,  $\varphi_C=0$, while perfect agreement or disagreement yields  $\varphi_C=1$ or $\varphi_C=-\dfrac{1}{2}$ in the bivariate case (\cite{nelsen2006introduction}).

While Spearman's footrule has historically been underutilized due to limited appreciation of its statistical properties, recent studies have underscored its distinct advantages and practical utility. These merits are fourfold: (1) Computational Simplicity and Efficiency. Spearman's footrule calculates the sum of absolute differences between ranks, requiring only \(O(n)\) time complexity. This makes it highly efficient for large-scale data compared to methods like Kendall's tau, which requires counting concordant/discordant pairs (\(O(n^2)\)).  (2) Sensitivity to Positional Differences. Spearman's footrule directly quantifies the absolute displacement of ranks, emphasizing the ``magnitude of rank shifts''. This sensitivity is critical in applications like search engine ranking evaluation, where positional accuracy at the top matters most. In contrast, classical Kendall's tau measures "order consistency", and Spearman's rho captures linear relationships between ranks, making it difficult for these two to achieve the same level of positional sensitivity. (3) Intuitive Interpretation.  The metric reflects the total displacement of ranks, and its normalized form (see Equation (\ref{Spearman_footrule_coefficient}) in Section \ref{section2}) provides a clear measure of alignment between rankings, making it accessible to non-expert users. (4) Robustness to Outliers. Relying on its rank-based nature and the use of absolute difference distance, Spearman's footrule exhibits insensitivity to extreme outliers compared to Spearman's rho. Despite both being rank-based methods, Spearman's rho's Euclidean distance feature results in weaker robustness than Spearman's footrule. These strengths render Spearman's footrule broadly applicable across multiple domains. In genomics, \cite{kim2004spearman} proposed a function of Spearman's footrule to assess reproducibility in microarray experiments, which are prone to generating outliers due to low signal-to-noise ratios. In information retrieval, it quantifies discrepancies between ranked lists \cite{fagin2003comparing,mikki2010comparing}. \cite{iorio2009identifying} and \cite{lin2009integration} applied the same idea in gene expression profiling and bioinformatics studies. Furthermore, preference learning frameworks incorporate the footrule distance into Bayesian Mallows models for aggregating incomplete rankings and quantify uncertainties in consensus rankings \cite{vitelli2018probabilistic}.

Despite the significant advantages of rank-based statistics in applications, the inherent dependence structure of ranks has historically complicated their theoretical advancement. Seeking asymptotic representation is an important technical means in simplifying rank-based statistical inference. The most common example is linear rank statistics (Section 13.1 of the classic monograph in statistics (\cite{van2000asymptotic})). The class of simple linear rank statistics is sufficiently large to contain interesting statistics for testing a variety of hypotheses, such as the Wilcoxon test statistic, van der Waerden test statistic, Median test statistic, Log rank test statistic, etc. In Theorem 13.5 of \cite{van2000asymptotic}, they sought an asymptotic representation for a family of linear rank statistics, which are composed of independent and identically uniformly distributed random variables. Their proof utilized the martingale convergence theorem and H\'{a}jek projection technique. Furthermore, in their Corollary 13.8, this asymptotic representation was used to prove the asymptotic normality of linear rank statistics. Since this asymptotic representation is a sum of independent variables, it greatly simplifies the proof of asymptotic normality. In their subsequent chapters, this asymptotic representation technique was used to prove that these linear rank statistics are asymptotically efficient within the class of all tests. Not only that, but there are also other nonlinear rank statistics. For example, \cite{angus1995coupling} used coupling techniques to obtain the asymptotic representation of the rank statistic $B_{n}=\sum_{k=1}^{n-1}\left|\pi_{k}-\pi_{k+1}\right|$, where $\left(\pi_{1}, \pi_{2}, \ldots, \pi_{n}\right)$  is a random permutation of the integers  $1,2, \ldots, n $. This asymptotic representation also includes independent and uniformly distributed random variables and was used to prove its asymptotic normality. Later, this rank statistic was used by \cite{chatterjee2021new} to construct the recently popular Chatterjee's rank correlation coefficient. In \cite{shi2022power}, this asymptotic representation was further used to study the power analysis of Chatterjee's rank correlation coefficient. In addition, recent \cite{lin2023boosting} and \cite{xia2024improved} have improved  this correlation coefficient, similarly using asymptotic representations (see their Remark 10 and Section 3, respectively) to establish the relevant asymptotic theory of the statistics and perform hypothesis testing.

For Spearman's footrule correlation coefficient, several studies have been conducted on its theory. For instance, early work by \cite{diaconis1977spearman} established the asymptotic normality of Spearman's footrule correlation coefficient under independence using combinatorial arguments. Subsequent studies, such as \cite{sen1983spearman}, leveraged Markov chain properties and martingale theory to derive similar results, emphasizing its significance in permutation-based frameworks. Despite these advances, critical rank-based gaps persist. To address this problem, we derive two distinct asymptotic representations under the null hypothesis of independence, which are also composed of independent and identically uniformly distributed random variables and do not depend on the original data distribution. Therefore, they do not disrupt the distribution-free property of tests based on Spearman's footrule rank correlation coefficient. Our motivations for seeking asymptotic representations of Spearman's footrule correlation coefficient are similar to the previous analysis, mainly including the following points: First, it is used to simplify the proof of the limiting null distribution of Spearman's footrule (Theorem \ref{the_limiting_null_distribution} in current paper). Although the previous two literatures have done similar things, we use different ideas. Second, it is used to extend to the multivariate Spearman's footrule correlation coefficient and obtain its asymptotic theory using asymptotic representations (this motivation has been developed into another achievement and is not presented here). Third, it is used to study the power analysis of Spearman's footrule and non-parametric confidence intervals, which will be explored as a future research direction.

The proof routes of the two proposed asymptotic representations are different from those of the previous related literature. Specifically, by replacing the empirical distribution functions with their population counterparts, we establish an initial asymptotic representation for Spearman's footrule through the empirical process. This approach circumvents the complexities introduced by rank dependencies, directly linking the statistic to its limiting behavior. Building on the first result, the H\'{a}jek projection technique is further employed to decompose Spearman's footrule into a linear combination of independent components. This decomposition not only reinforces the asymptotic normality conclusion but also elucidates the role of rank transformations in the statistical structure.

\section{Asymptotic representations}\label{section2}
At the outset of the article, to enhance clarity and streamline the subsequent notation, we first pre-compiled a notation table presenting Spearman's footrule correlation coefficient and its first and second representations under the independence of $X$ and $Y$. This table systematically documents their corresponding locations within the text, alongside the associated expectations and variances under the independence of $X$ and $Y$.
\begin{table}[htbp]
	\centering
	\caption{Notation summary table for Spearman's footrule  correlation coefficient and its first and second asymptotic representations.}
	\begin{tabular}{lcccc}
		\toprule
		& Notation & Location & Expectation & Variance \\
		\midrule
		Spearman's footrule correlation coefficient & 	$\varphi_n$     & Formula (\ref{Spearman_footrule_coefficient})     & 0     & $\frac{2n^2+7}{5(n+1)(n-1)^2}$ \\
		The first asymptotic representation & $\varphi_n^{\prime}$     & Formula (\ref{asymptotic_representation1})     & 0     & $\frac{2n^2}{5(n+1)^2(n-1)}$ \\
		The second asymptotic representation & $\varphi_n^{\prime\prime}$     & Formula (\ref{asymptotic_representation2})     & 0     & $\frac{2n}{5(n+1)^2}$\\
		\bottomrule
	\end{tabular}%
	\label{table_notation}%
\end{table}%

In this context, the joint distribution function of the bivariate continuous random variable $(X, Y)$ is denoted by $P(x, y)$, and their respective marginal distribution functions are represented by $F(x)$  and  $G(y)$. A finite sample of size $n$, comprising $\{\left(X_{1}, Y_{1}\right), \ldots,\left(X_{n}, Y_{n}\right)\} $, is obtained independently and identically distributed (i.i.d.) from $(X, Y)$. Let $R_i=\sum_{k=1}^n\mathbb{I}(X_k\leqslant X_i) $ be the rank of $X_i$ with indicator function $\mathbb{I}(\cdot) $, $i=1,...,n$. Similarly, $S_i=\sum_{k=1}^n\mathbb{I}(Y_k\leqslant Y_i)$ is the rank of $Y_i$.  Then, Spearman's footrule rank correlation coefficient is given by
\begin{eqnarray}\label{Spearman_footrule_coefficient}
	\varphi_n:=\varphi\left(\{(X_i,Y_i)\}_{i=1}^{i=n} \right)= 1-\dfrac{3}{n^2-1}\sum_{i=1}^n|R_i-S_i|.
\end{eqnarray}
Under the assumption of independence between  $X$ and $Y$, its expectation and variance are as follows
$$\mathrm{E}\varphi_n=0, \quad \mathrm{Var}(\varphi_n)=\dfrac{2n^2+7}{5(n+1)(n-1)^2}. $$
Although the existence of ranks makes the tests based on $\varphi_n$ fully distribution-free, i.e., not rely on the underlying distribution of the data, the dependence among ranks in practical applications complicates the derivation of certain asymptotic theories under independence between  $X$ and $Y$. Below, we introduce two asymptotic representations of $\varphi_n$ to address this issue. Through intuitive and straightforward calculation, $\varphi_n$ can be rewritten as
\begin{eqnarray*}
	\varphi_n=\dfrac{3n^2}{n^2-1}\left(\dfrac{1}{n^2}\sum_{i=1}^n \sum_{j=1}^n|   F_n(X_i)-G_n(Y_i) |-\dfrac{1}{n}\sum_{i=1}^n|F_n(X_i)-G_n(Y_i) | \right),
\end{eqnarray*}
which is composed of components that involve the empirical distribution functions $F_n(x)=\frac{1}{n}\sum_{k=1}^n\mathbb{I}(X_k\leqslant x)$ and $G_n(y)=\frac{1}{n}\sum_{k=1}^n\mathbb{I}(Y_k\leqslant y)$ of $X$ and $Y$ for any $x\in \mathbb{R}$ and $y\in \mathbb{R}$. A natural inclination is to replace these two empirical functions with their population counterparts $F$ and $G$, but there are still remaining terms that need to be addressed. Notably, $F(X)$ and $ G(Y)$ follow a uniform distribution over the interval $[0,1]$. This ultimately induces the following theorem. The specific proof involving empirical processes, is presented in Appendix \ref{appendixA}.

\begin{theorem}[\textbf{The first asymptotic representation}]\label{representation1}
	Under the assumption of independence between $X$ and $Y$, $\varphi_n$ is asymptotically identically distributed with the following form,
	\begin{eqnarray}\label{asymptotic_representation1}
		\varphi_n^{\prime}&=&\dfrac{3n^2}{n^2-1}\left(\dfrac{1}{n^2}\sum_{i=1}^n \sum_{j=1}^n|U_i-V_j|-\dfrac{1}{n}\sum_{i=1}^n|U_i-V_i| \right) ,
	\end{eqnarray}
	where, $U_1, ..., U_n$ and $V_1, ..., V_n$  are i.i.d.  random variables from uniform distribution $U(0,1)$, and $U_i$ and $V_i$ are also independent for $i=1,\cdots,n$. 	Additionally,
	$$\mathrm{E}\varphi_n^{\prime}=0,\quad \mathrm{Var}(\varphi_n^{\prime})=\dfrac{2n^2}{5(n+1)^2(n-1)}. $$
	
\end{theorem}
To further obtain a simpler form, we will now apply the H\'{a}jek projection to $\varphi_n^{\prime}$, resulting in the following theorem.
\begin{theorem}[\textbf{The second asymptotic representation}]\label{representation2}
	Under the assumption of independence between $X$ and $Y$,  $\varphi_n^{\prime}$'s H\'{a}jek asymptotic representation is as follows
	\begin{eqnarray}\label{asymptotic_representation2} \varphi_n^{\prime\prime}=\dfrac{3}{n+1}\sum_{i=1}^n\left( \dfrac{2}{3}-|U_i-V_i|-U_i(1-U_i)-V_i(1-V_i)\right) ,
	\end{eqnarray}
	with expectation and variance,
	$$\mathrm{E}\varphi_n^{\prime\prime}=0,\quad \mathrm{Var}(\varphi_n^{\prime\prime})=\dfrac{2n}{5(n+1)^2}. $$
	
\end{theorem}
One significant application of the asymptotic representations developed in this study is to establish the asymptotic normality of $\varphi_n$ under the independence condition between $X$ and $Y$. By utilizing Theorem \ref{representation2} in conjunction with Theorem \ref{representation1}, the limiting null distribution of $\varphi_n$ can be readily obtained.

\begin{theorem}[\textbf{The limiting null distribution}] \label{theorem2.3}\label{the_limiting_null_distribution}
	Under the assumption of independence between $X$ and $Y$, $\sqrt{n}\varphi_n$, $ \sqrt{n}\varphi_n^{\prime}$ and  $\sqrt{n} \varphi_n^{\prime\prime}$ converge weakly to  the same normal distribution with mean 0 and  variance $\dfrac{2}{5}$.
\end{theorem}
\begin{remark}
	In the existing literature, there are various approaches for deriving the limiting null distribution of $\sqrt{n}\varphi_n$. \cite{diaconis1977spearman} established its normality by utilizing the combinatorial central limit theorem developed by \cite{hoeffding1951combinatorial}. In \cite{sen1983spearman}, martingale techniques are incorporated into the study of the asymptotic normality of Spearman's footrule. \cite{shi2023max} derived the rate of convergence for the standardized Spearman's footrule to the standard normal distribution based on the combinatorial central limit theorem and the Cram\'{e}r-type moderate deviation result (Section 4.1 of \cite{chen2013stein}). \cite{shi2025testing} obtained an alternative form of the convergence rate using the Edgeworth expansion \cite{small2010expansions}, and these two results also naturally led to the limiting null distribution of Spearman's footrule. It is evident that our approach differs significantly from these methods and serves as the basis for further  theoretical investigations into Spearman's footrule.
\end{remark}

\section{Simulation studies}

In this section, we mainly evaluate the performance of the two proposed asymptotic representations using Monte Carlo simulations from the following two aspects. In the first subsection (Section \ref{section3.1}), we examine the estimated means and variances of Spearman's footrule correlation coefficient, the two asymptotic representations ($\varphi_n$, $\varphi_n^{\prime}$, and $\varphi_n^{\prime\prime}$). In the second subsection (Section \ref{section3.2}), we investigate the approximation between $\varphi_n$, $\varphi_n^{\prime}$ and $\varphi_n^{\prime\prime}$,  as well as their approximation to the normal limit distribution. For the calculation of $\varphi_n$, let $X$ and $Y$ be drawn from the standard normal distribution and the standard uniform distribution, respectively. For $\varphi_n^{\prime}$ and $\varphi_n^{\prime\prime}$,  their calculations are performed by generating random numbers from the standard uniform distribution according to Equations (\ref{asymptotic_representation1}) and (\ref{asymptotic_representation2}).

\subsection{Simulation of estimated means and variances}\label{section3.1}
For the proposed asymptotic representations, serving as estimators of the population form $\varphi_C$ (which has been presented in the first paragraph of Section \ref{section1}), it is natural to consider their estimated means and variances. Under the scenario where $X$ and $Y$ are independent, the true value of the population is 0. In this study, we employ the common estimated mean (EM), estimated variance (EV), bias, and root mean square error (RMSE) for simulations. We select sample sizes of $n = 10, 20, 30, \cdots, 100$ and set the number of simulations to 10,000. The quantitative results are summarized in Table \ref{table_EV_EM}. From these results, it can be observed that the estimated means and variances closely approximate the true means and variances (the true variances are provided in Table \ref{table_notation}). Regarding the evaluation of RMSE, it is evident that the RMSE of all methods decreases as the sample size increases, and the RMSE of the two asymptotic representations is smaller than that of $\varphi_n$. This demonstrates that both our proposed methods and the original Spearman's footrule are suitable as estimators, and our proposed methods are superior to $\varphi_n$.

Regarding the bias, although it is relatively close to the true value of 0, we need to visualize its trend. Therefore, we further increase the sample size to $n = 1000$, while keeping other settings unchanged, and plot the Bias and RMSE curves as shown in Figure \ref{figure_EV_EM}. From the figure, it can be seen that the biases of $\varphi_n$, $\varphi_n^{\prime}$, and $\varphi_n^{\prime\prime}$ all tend to 0 with increasing sample size, albeit in a fluctuating manner.  For the visualization of RMSE, only when the sample size is small, the corresponding value of $\varphi_n$ is slightly larger. Subsequently, all three methods exhibit almost identical trends towards 0, suggesting that their asymptotic performances are nearly similar when the sample size is large.

\begin{table}[htbp]
	\centering
	\addtolength{\tabcolsep}{-4.5pt}
	\caption{The estimated means(EM), estimated variances (EV), biases, and root mean square errors (RMSE) of $\varphi_n$, $\varphi_n^{\prime}$, and $\varphi_n^{\prime\prime}$.}
	\begin{tabular}{cccccccccccc}
		\toprule
		&       & $n=10$  & $n=20$  & $n=30$  & $n=40$  & $n=50$  & $n=60$  & $n=70$  & $n=80$  & $n=90$  & $n=100$ \\
		\midrule
		\multirow{4}[1]{*}{$\varphi_n$} & EM & -0.00038  & 0.00051  & -0.00013  & -0.00170  & 0.00019  & -0.00144  & -0.00094  & -0.00014  & -0.00008  & 0.00066  \\
		& EV & 0.04652  & 0.02079  & 0.01389  & 0.01018  & 0.00804  & 0.00673  & 0.00589  & 0.00504  & 0.00449  & 0.00394  \\
		& Bias  & -0.00038  & 0.00051  & -0.00013  & -0.00170  & 0.00019  & -0.00144  & -0.00094  & -0.00014  & -0.00008  & 0.00066  \\
		& RMSE  & 0.21568  & 0.14418  & 0.11783  & 0.10088  & 0.08966  & 0.08202  & 0.07673  & 0.07102  & 0.06701  & 0.06280  \\
		&       &       &       &       &       &       &       &       &       &       &  \\
		\multirow{4}[0]{*}{$\varphi_n^{\prime}$} & EM & 0.00225  & 0.00213  & -0.00020  & 0.00001  & -0.00098  & 0.00074  & -0.00045  & -0.00040  & 0.00052  & -0.00033  \\
		& EV & 0.03677  & 0.01965  & 0.01290  & 0.00984  & 0.00779  & 0.00657  & 0.00572  & 0.00497  & 0.00448  & 0.00402  \\
		& Bias  & 0.00225  & 0.00213  & -0.00020  & 0.00001  & -0.00098  & 0.00074  & -0.00045  & -0.00040  & 0.00052  & -0.00033  \\
		& RMSE  & 0.19177  & 0.14017  & 0.11357  & 0.09921  & 0.08824  & 0.08106  & 0.07561  & 0.07049  & 0.06693  & 0.06340  \\
		&       &       &       &       &       &       &       &       &       &       &  \\
		\multirow{4}[1]{*}{$\varphi_n^{\prime\prime}$} & EM & 0.00122  & 0.00083  & -0.00014  & 0.00204  & 0.00083  & -0.00030  & 0.00032  & 0.00032  & 0.00083  & 0.00047  \\
		& EV & 0.03237  & 0.01781  & 0.01252  & 0.00971  & 0.00786  & 0.00645  & 0.00552  & 0.00497  & 0.00428  & 0.00389  \\
		& Bias  & 0.00122  & 0.00083  & -0.00014  & 0.00204  & 0.00083  & -0.00030  & 0.00032  & 0.00032  & 0.00083  & 0.00047  \\
		& RMSE  & 0.17991  & 0.13346  & 0.11190  & 0.09858  & 0.08865  & 0.08030  & 0.07426  & 0.07048  & 0.06542  & 0.06238  \\
		\bottomrule
	\end{tabular}%
	\label{table_EV_EM}%
\end{table}%

\begin{figure}[!htb]
	\centering
	\includegraphics[width=1.0\columnwidth]{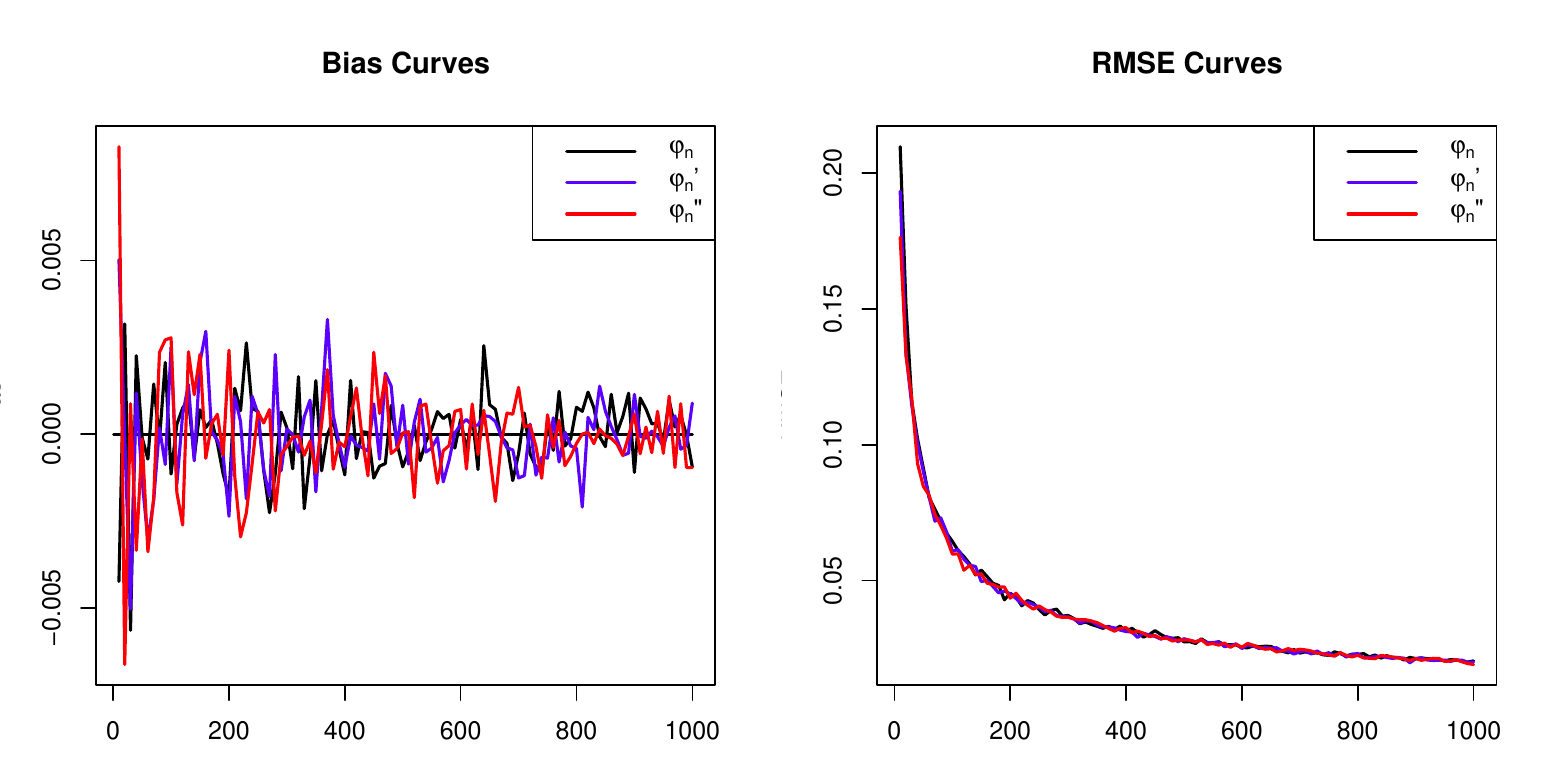}
	\caption{The bias and root mean square error (RMSE) curves of $\varphi_n$, $\varphi_n^{\prime}$, and $\varphi_n^{\prime\prime}$.}
	\label{figure_EV_EM}
\end{figure}

\subsection{Simulation of the normal limiting distribution}\label{section3.2}

In this subsection, we simulate the asymptotic behaviors of $\sqrt{n}\varphi_n$, $ \sqrt{n}\varphi_n^{\prime}$ and  $ \sqrt{n}\varphi_n^{\prime\prime}$ in three ways. Specifically, for the first two ways, we estimate their empirical density functions and cumulative distribution functions (CDF) through simulations. Here, we set four sample sizes, namely $n = 10, 20, 30$, and $100$, and run 100,000 simulations. For the third way, we employ the Kolmogorov-Smirnov (KS) two sample test to examine the identicalness of the distributions between each pair of the proposed methods and between each method and the normal distribution \cite{schroer1995exact}. There are six combinations in total, including ($\varphi_n-N(0,0.4)$, $\varphi_n^{\prime}-N(0,0.4)$, $\varphi_n^{\prime\prime}-N(0,0.4)$, as well as $\varphi_n-\varphi_n^{\prime}$, $\varphi_n-\varphi_n^{\prime\prime}$, $\varphi_n^{\prime}-\varphi_n^{\prime\prime}$), where $N(0,0.4)$ represents the limiting null distribution of $\sqrt{n}\varphi_n$, $ \sqrt{n}\varphi_n^{\prime}$ and  $ \sqrt{n}\varphi_n^{\prime\prime}$. The sample size is set to $n = 10, 20, 30, ..., 100$, and the number of simulations is 1000. The KS test is performed using the "ks.test" function from the standard library in R software. The empirical density functions and cumulative distribution function curves are shown in Figure \ref{densitycurve} and Figure \ref{cdfcurve}, respectively. All $p$-values from the KS tests are presented in Table \ref{kstest}.

As can be seen from Figure \ref{densitycurve} and Figure \ref{cdfcurve}, even with a sample size of $n=30$, the distributions of  $ \sqrt{n}\varphi_n^{\prime}$ and  $\sqrt{n} \varphi_n^{\prime\prime}$, as well as  $\sqrt{n}\varphi_n$, are very close to the theoretical limiting distribution (the normal distribution with mean 0 and  variance $\dfrac{2}{5}$). However, the best approximation is provided by $\sqrt{n} \varphi_n^{\prime}$, followed by $ \sqrt{n}\varphi_n^{\prime\prime}$, and the worst by $\sqrt{n}\varphi_n$. This also reflects, to some extent, the rate at which the distributions of these three representations converge to the limiting distribution. It is worth noting that when the sample size is extremely small ($n=10$), the performance of $\sqrt{n}\varphi_n$ is not very good as shown in the first subfigures of Figure \ref{densitycurve} and Figure \ref{cdfcurve}. This is due to the permutations of ranks present in the structure of $\varphi_n$. Despite these permutations being different and numerous (factorial of 10), the calculated values of $\sqrt{n}\varphi_n$ exhibit a large number of repetitions. Even with a very large number of simulation repetitions (100,000), there are relatively few distinct values (only a few dozen). These fewer discrete values ultimately lead to the non-smooth curve of the kernel density estimation for $ \sqrt{n}\varphi_n$ in Figure \ref{densitycurve} when $ n=10$, as well as the stepwise appearance of the empirical cumulative distribution function for $ \sqrt{n}\varphi_n$ in Figure \ref{cdfcurve} when $n=10$. However, as the sample size increases, this phenomenon gradually disappears. The asymptotic behaviors of all methods become similar and approach the theoretical normal limiting distribution.

The $p$-values presented in Table \ref{kstest} indicate that, except for the case with a small sample size ($n=10$), each pair of methods is approximately identically distributed. Furthermore, the two proposed asymptotic representations and the original Spearman's footrule correlation coefficient all exhibit approximate normal distributions. For the case of $n=10$, as long as  $\varphi_n$ is not involved, the above conclusion remains valid. However, for the KS test of $\varphi_n$, extremely small $p$-values are observed, suggesting that for these scenarios, there is sufficient evidence to conclude that the distribution of $\varphi_n$ differs from that of other methods and also deviates significantly from the normal distribution. This finding is consistent with the previous analysis.

\begin{figure}[!htb]
	\centering
	\includegraphics[width=1.0\columnwidth]{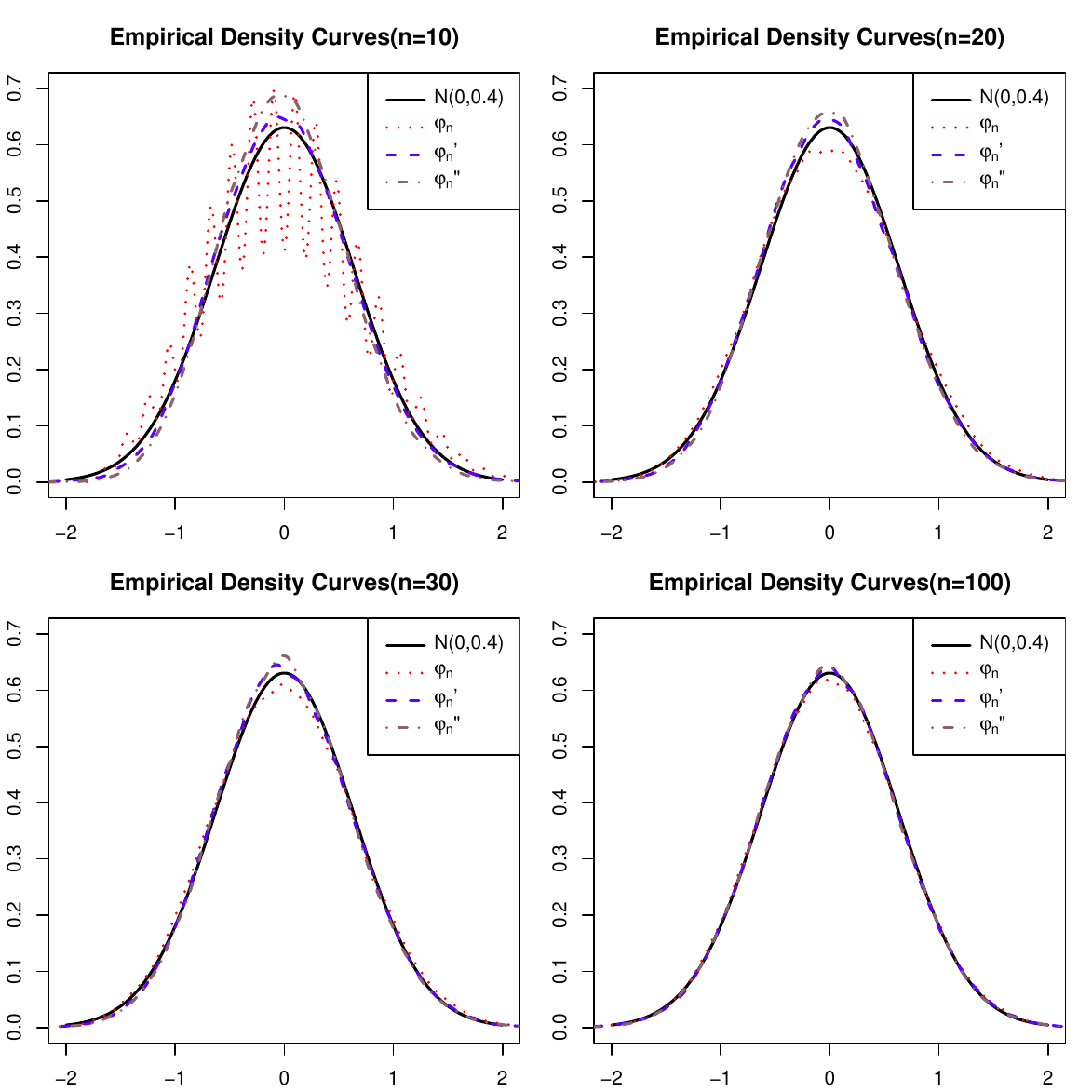}
	\caption{Empirical density curves of $\sqrt{n}\varphi_n$, $ \sqrt{n}\varphi_n^{\prime}$ and  $\sqrt{n} \varphi_n^{\prime\prime}$, estimated using kernel density estimation with a Gaussian kernel, where the solid line represents a normal curve with a mean of 0 and a variance of 0.4. }
	\label{densitycurve}
\end{figure}

\begin{figure}[!htb]
	\centering
	\includegraphics[width=1.0\columnwidth]{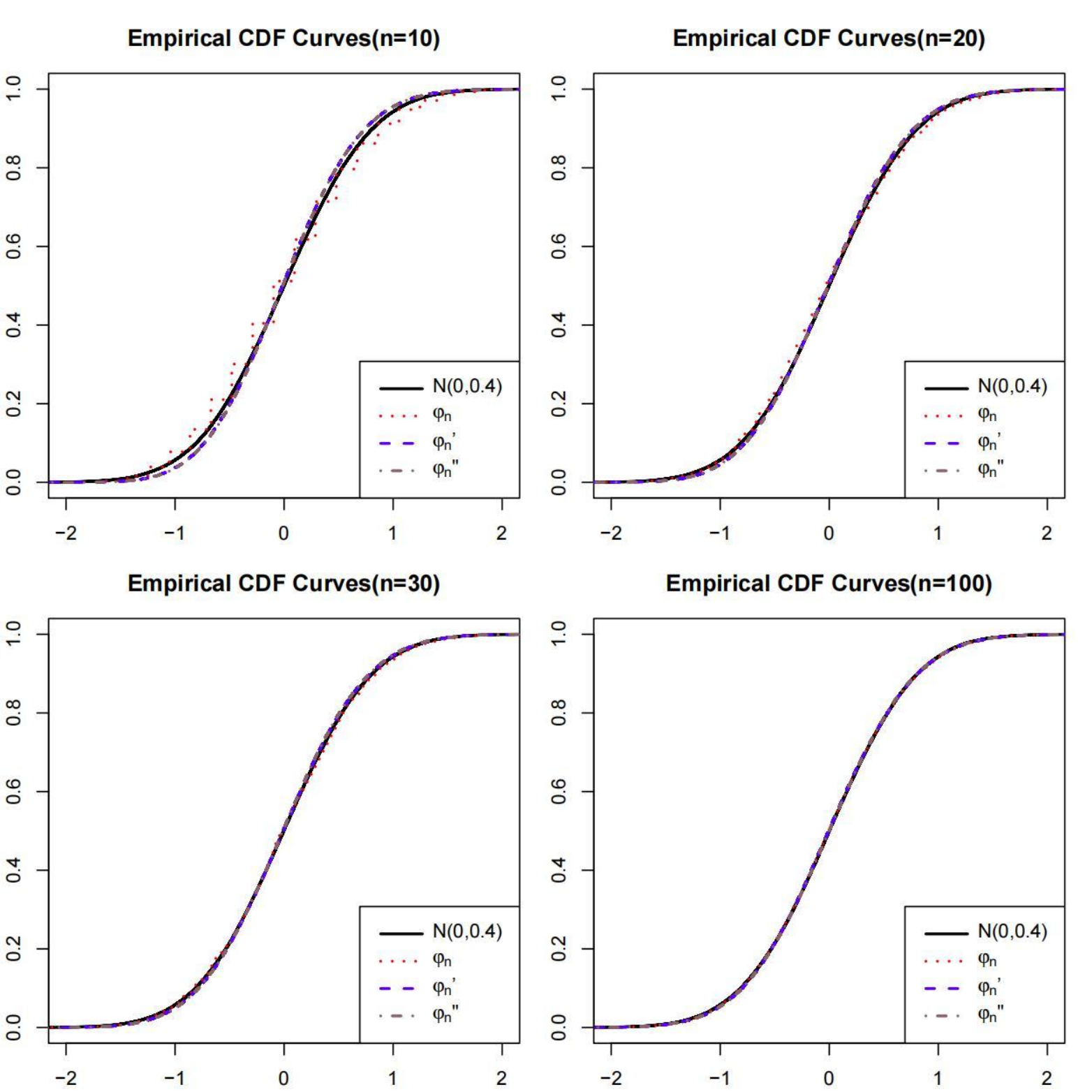}
	\caption{Empirical CDF curves of $\sqrt{n}\varphi_n$, $ \sqrt{n}\varphi_n^{\prime}$ and  $\sqrt{n} \varphi_n^{\prime\prime}$. The solid line represents a normal cumulative distribution function curve with a mean of 0 and a variance of 0.4.}
	\label{cdfcurve}
\end{figure}

\begin{table}[htbp]
	\centering
	\addtolength{\tabcolsep}{-3.5pt}
	\caption{The $p$-values of KS test for six combinations.}
	\begin{tabular}{ccccccccccc}
		\toprule
		& $n=10$  & $n=20$  & $n=30$  & $n=40$  & $n=50$  & $n=60$  & $n=70$  & $n=80$  & $n=90$  & $n=100$ \\
		\midrule
		$\varphi_n-N(0,0.4)$    & 0.00003  & 0.03328  & 0.08691  & 0.18112  & 0.18112  & 0.34100  & 0.10828  & 0.50036  & 0.36998  & 0.02246  \\
		$\varphi_n^{\prime}-N(0,0.4)$    & 0.31358  & 0.31358  & 0.14834  & 0.93558  & 0.79439  & 0.31358  & 0.60992  & 0.31358  & 0.64756  & 0.26338  \\
		$\varphi_n^{\prime\prime}-N(0,0.4)$    & 0.26338  & 0.68523  & 0.75910  & 0.40047  & 0.46577  & 0.85929  & 0.64756  & 0.09710  & 0.95406  & 0.43243  \\
		&       &       &       &       &       &       &       &       &       &  \\
		$\varphi_n-\varphi_n^{\prime}$    & 0.00103  & 0.04282  & 0.43243  & 0.24060  & 0.21933  & 0.46577  & 0.75910  & 0.98826  & 0.85929  & 0.31358  \\
		$\varphi_n-\varphi_n^{\prime\prime}$    & 0.00000  & 0.09710  & 0.07762  & 0.21933  & 0.03328  & 0.53605  & 0.43243  & 0.40047  & 0.34100  & 0.53605  \\
		$\varphi_n^{\prime}-\varphi_n^{\prime\prime}$    & 0.31358  & 0.72255  & 0.72255  & 0.82796  & 0.34100  & 0.13383  & 0.53605  & 0.36998  & 0.72255  & 0.68523  \\
		\bottomrule
	\end{tabular}%
	\label{kstest}%
\end{table}%

\section{Conclusions}\label{section6}
Spearman's footrule, despite its robustness, faces theoretical complexities due to rank dependencies. Two asymptotic representations address this issue under independence. The initial representation simplifies the statistic by using population distribution functions. The subsequent use of H\'{a}jek projection decomposes the footrule into independent components, reinforcing the asymptotic normality, thus enhancing its theoretical understanding.

\appendix
\section{Appendix}\label{appendixA}
\begin{lemma}\label{lemmaA1}
	Given that $U_1$, $V_1$, and $V_2$ are independently and identically distributed from the uniform distribution $U(0,1)$, through simple integral calculation, the following facts can be easily deduced:\\
	$\mathrm{E}|U_1-V_1|=\dfrac{1}{3} $, $\mathrm{E}\left( |U_1-V_1||U_1\right) =\dfrac{1}{2}-U_1(1-U_1) $.
	$\mathrm{E}(U_1 (1 -U_1 ))=\dfrac{1}{6}$, $ \mathrm{Var}(|U_1 -V_1 |)=\dfrac{1}{18} $,
	$ \mathrm{Var}(U_1 (1 -U_1 ))=\dfrac{1}{180}$, $ \mathrm{Cov}(|U_1 -V_1 |,U_1 (1 -U_1 ))=-\dfrac{1}{180}$, $ \mathrm{Cov}(|U_1 -V_1 |,|U_1 -V_2 |)=\dfrac{1}{180}$.
\end{lemma}
\begin{lemma}[Lemma 19.24 in \cite{van2000asymptotic}]\label{lemmaA2}
	Suppose that  $\mathcal{F}$  is a  $P$ -Donsker class of measurable functions and  $f_{n}$  is a sequence of random functions that take their values in $ \mathcal{F}$  such that  $\int\left(f_{n}(x)-f(x)\right)^{2} d P(x)$  converges in probability to 0 for some  $f\in L_{2}(P)$ . Then  $\mathbb{G}_{n}\left(f_{n}-f\right) \xrightarrow{P} 0$  and hence  $\mathbb{G}_{n} f_{n} \rightsquigarrow \mathbb{G}_{P} f $.
\end{lemma}

\begin{proof}[\textbf{Proof of Theorem \ref{representation1}}]
	Let $ (X_1,Y_1), \ldots, (X_n,Y_n)\in  \mathcal{Z}=\mathbb{R} \times  \mathbb{R} $  be a random sample from a probability distribution  $P$  defined on a measurable space  $(\mathcal{Z}, \mathcal{A})$. We denote two empirical distributions as $ \mathbb{P}_{n}=n^{-1} \sum_{i=1}^{n} \delta_{(X_i,Y_i)} $ and $ \mathbb{P}_{n}^\prime=n^{-2} \sum_{i=1}^{n} \sum_{j=1}^{n} \delta_{(X_i,Y_j)} $, where  $\delta_{(x,y)}$  represents the probability distribution degenerate at the point  $(x,y)$. For a given measurable function $ f: \mathcal{Z} \mapsto \mathbb{R} $, we  use $ \mathbb{P}_{n} f$ and $ \mathbb{P}_{n}^\prime f$ to denote  the expectations of  $f$  under the empirical measures $\mathbb{P}_{n}$ and $\mathbb{P}_{n}^\prime$, respectively. Similarly, $ P f$  represents the expectation of $f$ under  $P(x,y)=F(x)G(y)$. Thus, the expressions are given by:
	\begin{eqnarray}\label{Pnf_Pf}
		\mathbb{P}_{n} f=\frac{1}{n} \sum^{n}_{i=1} f\left(X_i,Y_i\right),\mathbb{P}^{\prime}_{n} f=\frac{1}{n^2} \sum^{n}_{i=1} \sum^{n}_{j=1} f\left(X_i,Y_j\right), \quad P f=\int f d P.
	\end{eqnarray}
	We choose
	\begin{eqnarray}\label{f_fn}
		f(x,y)=|F(x)-G(y)| \ \    \text{and } \ \ f_n(x,y)=|F_n(x)-G_n(y)|,
	\end{eqnarray}
	where $F_n(x)=\frac{1}{n}\sum_{k=1}^n\mathbb{I}(X_k\leqslant x)$ and $G_n(y)=\frac{1}{n}\sum_{k=1}^n\mathbb{I}(Y_k\leqslant y)$.
	
	Let $\mathcal{F}_0$ be the collection of cumulative distribution functions for all univariate continuous variables and   $\mathcal{F}=\{|F(x)-G(y)|\in\mathbb{R}:F,G\in\mathcal{F}_0 \text{ and } x,y \in  \mathbb{R}\}$. Example 19.6 of \cite{van2000asymptotic} illustrates that $\mathcal{F}_0$ is a $P$-Donsker class. Thus, for all $(x,y) \in \mathbb{R}^2$ and $F, G \in \mathcal{F}_0$, according to the information provided on page 19 of \cite{kosorok2008introduction}, along with the fact that $|a-b|=\mathrm{max}\{a,b\}-\mathrm{min}\{a,b\}$,  $\mathcal{F}$ is also a $P$-Donsker class.

	By the law of large numbers, for every $x$ and $y$, it is apparent that $\mathrm{Sup}_{x\in R}|F_n(x)-F(x)|\xrightarrow{a.s.}0$ and $\mathrm{Sup}_{y\in R}|G_n(y)-G(y)|\xrightarrow{a.s.}0$ hold, thus resulting in $\mathrm{Sup}_{x\in R,y \in R}|f_n(x,y)-f(x,y)|\xrightarrow{a.s.}0$, hence for some  $f \in L_{2}(P) $, $\int(f_{n}(x)-f)^{2}d P \xrightarrow{p}0$ follows.
	Then, define
	$$\mathbb{G}_n:=\sqrt{n}(\mathbb{P}_{n}-P),\ \  \mathbb{G}^\prime_n:=\sqrt{n}(\mathbb{P}_{n}^\prime-P). $$
	The empirical process evaluated at $f$ is  $$\mathbb{G}_nf=\sqrt{n}(\mathbb{P}_{n}f-Pf),\ \  \mathbb{G}_n^\prime f=\sqrt{n}(\mathbb{P}_{n}^\prime f-Pf). $$

	Thus, based on previous analysis, by directly applying Lemma \ref{lemmaA2} (see also Lemma 19.24 in \cite{van2000asymptotic}), one has
	$$\mathbb{G}_n^\prime\left(  f_n-f \right)=o_p(1),\ \  \mathbb{G}_n\left(  f_n-f \right)=o_p(1), $$
	i.e.,
	$$\left( \mathbb{P}_{n}^\prime-P\right) \left(  f_n-f \right)=o_p(n^{-1/2}),\ \ \left( \mathbb{P}_{n}-P\right) \left(  f_n-f \right)=o_p(n^{-1/2}). $$
	Further expanding these two expressions leads to the following forms,
	$$\mathbb{P}_{n}^\prime f_n-\mathbb{P}_{n}^\prime  f-P f_n+Pf =o_p(n^{-1/2}),\ \ \mathbb{P}_{n} f_n-\mathbb{P}_{n}  f-P f_n+Pf =o_p(n^{-1/2}). $$
	Note that the combination of equations (\ref{Pnf_Pf}) and (\ref{f_fn}) can yield the following forms:
	$$\mathbb{P}_{n}^\prime f_n=\dfrac{1}{n^2}\sum_{i=1}^n\sum_{j=1}^n\left|F_n(X_i)-G_n(Y_j) \right|,\ \ \mathbb{P}_{n} f_n=\dfrac{1}{n}\sum_{i=1}^n\left|F_n(X_i)-G_n(Y_i) \right|.$$
	
	$$\mathbb{P}_{n}^\prime f=\dfrac{1}{n^2}\sum_{i=1}^n\sum_{j=1}^n\left|F(X_i)-G(Y_j) \right|=\dfrac{1}{n^2}\sum_{i=1}^n\sum_{j=1}^n\left|U_i-V_j \right|.$$
	$$\mathbb{P}_{n}f=\dfrac{1}{n}\sum_{i=1}^n\left|F(X_i)-G(Y_i) \right|=\dfrac{1}{n}\sum_{i=1}^n\left|U_i-V_i \right|.$$
	
	$$Pf_n=\mathrm{E}\left|F_n(X_i)-G_n(Y_i) \right|,\ \  Pf=\mathrm{E}\left|F(X_i)-G(Y_i) \right|=\mathrm{E}\left|U_i-V_i \right|.$$
	Thus,
	$$\dfrac{1}{n^2}\sum_{i=1}^n\sum_{j=1}^n\left|F_n(X_i)-G_n(Y_j) \right|=\dfrac{1}{n^2}\sum_{i=1}^n\sum_{j=1}^n\left|U_i-V_j \right|+Pf_n-Pf+o_p(n^{-1/2}),$$
	$$\dfrac{1}{n}\sum_{i=1}^n\left|F_n(X_i)-G_n(Y_i) \right|=\dfrac{1}{n}\sum_{i=1}^n\left|U_i-V_i \right|+Pf_n-Pf+o_p(n^{-1/2}).$$
	Combining the above equations, we obtain
	\begin{eqnarray*}
		\varphi_n&=&1-\dfrac{3}{n^2-1}\sum_{i=1}^n|R_i-S_i|\\
		&=&\dfrac{3}{n^2-1}\left( \dfrac{n^2-1}{3}-\sum_{i=1}^n|R_i-S_i|\right) \\
		&=&\dfrac{3}{n^2-1}\left(\dfrac{1}{n}\sum_{i=1}^n \sum_{j=1}^n|R_i-S_j|-\sum_{i=1}^n|R_i-S_i| \right)\\
		&=&\dfrac{3n^2}{n^2-1}\left(\dfrac{1}{n^2}\sum_{i=1}^n \sum_{j=1}^n|   F_n(X_i)-G_n(Y_i) |-\dfrac{1}{n}\sum_{i=1}^n|F_n(X_i)-G_n(Y_i) | \right)\\
		&=&\dfrac{3n^2}{n^2-1}\left(\dfrac{1}{n^2}\sum_{i=1}^n\sum_{j=1}^n\left|U_i-V_j \right|-\dfrac{1}{n}\sum_{i=1}^n\left|U_i-V_i \right|+o_p(n^{-1/2}) \right)\\
		&=&	\varphi_n^{\prime}+o_p(n^{-1/2}).
	\end{eqnarray*}

	Next, relying on the facts stated in Lemma \ref{lemmaA1}, we deal with the expectation and variance of $\varphi_n^{\prime}$. Let $C_1=\dfrac{1}{n^2}\sum_{i=1}^n\sum_{j=1}^n\left|U_i-V_j \right|$ and $C_2=\dfrac{1}{n}\sum_{i=1}^n\left|U_i-V_i \right|$, it is easy to calculate that
	$\mathrm{E}\varphi_n^{\prime}=0.$
	
	Furthermore, routine calculation yields
	\begin{eqnarray*}
		\mathrm{Var}(C_1)&=&\dfrac{1}{n^4}\left\lbrace n^2\mathrm{Var}(\left|U_1-V_1 \right|)+2\times n^2\times (n-1)\mathrm{Cov}(|U_1 -V_1 |,|U_1 -V_2 |) \right\rbrace \\
		&=&\dfrac{1}{n^4}\left\lbrace n^2 \times \dfrac{1}{18} +2 n^2 (n-1)\times\dfrac{1}{180}\right\rbrace\\
		&=&\dfrac{n+4}{90n^2}.\\
		\mathrm{Var}(C_2)&=&\dfrac{1}{n^2}\times n \mathrm{Var}(\left|U_1-V_1 \right|)=\dfrac{1}{18n}.\\
		\mathrm{Cov}(C_1,C_2)&=&\dfrac{1}{n^3}\left( \sum_{i=1}^n\sum_{j=1}^n\left|U_i-V_j \right|,\sum_{i=1}^n\left|U_i-V_i \right| \right) \\
		&=&\dfrac{1}{n^3}\left\lbrace n\mathrm{Var}(\left|U_1-V_1 \right|)+2(n-1) \mathrm{Cov}(|U_1 -V_1 |,|U_1 -V_2 |)\times n \right\rbrace \\
		&=& \dfrac{1}{n^3}\left\lbrace n \times\dfrac{1}{18}+2n(n-1)\times\dfrac{1}{180} \right\rbrace \\
		&=& \dfrac{n+4}{90n^2}.
	\end{eqnarray*}
	Ultimately, it is derived that
	\begin{eqnarray*}
		\mathrm{Var}(\varphi_n^{\prime})&=&\left( \dfrac{3n^2}{n^2-1}\right)^2\left[ \mathrm{Var}(C_1)+\mathrm{Var}(C_2) -2\mathrm{Cov}(C_1,C_2)\right] \\
		&=&\left( \dfrac{3n^2}{n^2-1}\right)^2\left(\dfrac{n+4}{90n^2}+ \dfrac{1}{18n}-2\times \dfrac{n+4}{90n^2} \right) \\
		&=& \dfrac{2n^2}{5(n+1)^2(n-1)}.
	\end{eqnarray*}
	
	The proof of Theorem \ref{representation1} is now complete.
\end{proof}

\begin{proof}[\textbf{Proof of Theorem \ref{representation2}}]
	Rewrite $\varphi_n^{\prime}$ as
	\begin{equation}\label{rewrite_widetilde_varphi}
		\begin{aligned}
			\varphi_n^{\prime} &=\dfrac{3}{n^2-1}\left(\sum_{i=1}^n \sum_{j=1}^n|U_i-V_j|-n\sum_{i=1}^n|U_i-V_i| \right)\\
			&= \dfrac{3}{n^2-1}\left(\sum_{i\neq j}^n |U_i-V_j|-(n-1)\sum_{i=1}^n|U_i-V_i| \right) \\
			&=\dfrac{3}{n^2-1}\left(\dfrac{n(n-1)}{2}T_1-(n-1)T_2\right),
		\end{aligned}
	\end{equation}
	where $T_1=\dfrac{2}{n(n-1)}\sum_{i\neq j}^n |U_i-V_j|$, and  $T_2=\sum_{i=1}^n|U_i-V_i| $.

	Since $T_2$ is already a sum of independent and identically distributed terms, its H\'{a}jek projection remains itself. Thus, we only need to calculate the H\'{a}jek representation for $T_1$.
	
	In fact, $T_1$ is a U-statistic that can be expressed as
	$$T_1=\dfrac{2}{n(n-1)}\sum_{i\neq j}^n |U_i-V_j|=\dfrac{2}{n(n-1)}\sum_{i< j}^n h\left(\left(U_i,V_i \right)^\top ,\left(U_j,V_j \right)^\top  \right),  $$
	where the symmetric kernel function is taken as
	$$ h\left(\left(u_1,v_1 \right)^\top ,\left(u_2,v_2 \right)^\top  \right)=|u_1-v_2|+|u_2-v_1|.$$
	It is evident that the variance of  $T_1$ exists. Let $\theta=\mathrm{E}\left[ h\left(\left(U_i,V_i \right)^\top ,\left(U_j,V_j \right)^\top  \right)\right] $, and  $h_1\left(\left(u,v \right)^\top \right)=\mathrm{E}\left[h\left(\left(u,v \right)^\top ,\left(U_2,V_2 \right)^\top  \right)\right]-\theta  $. According to Lemma  \ref{lemmaA1} and through simple derivation, we have  $\theta = E(|U_1 - V_2| + |U_2 - V_1|)=\dfrac{2}{3}$ and $h_1\left(\left(u,v \right)^\top \right) = \dfrac{1}{3} - u (1 - u)-v (1 - v)$.  The projection of $T_1 - \dfrac{2}{3}$ is then given by
	\begin{eqnarray*}
		\widetilde{T}_1:=\dfrac{2}{n}\sum_{i=1}^n\left[ \dfrac{1}{3} - U_i (1 - U_i)-V_i (1 - V_i) \right].
	\end{eqnarray*}
	Then, by applying Lemma 12.3 from \cite{van2000asymptotic}, we obtain
	$$ T_1 - \dfrac{2}{3}=\widetilde{T}_1+O_p(\dfrac{1}{n}),$$
	i.e.,
	$$ T_1 =\dfrac{2}{n}\sum_{i=1}^n\left[ \dfrac{1}{3} - U_i (1 - U_i)-V_i (1 - V_i) \right]+\dfrac{2}{3}+O_p(\dfrac{1}{n}).$$
	
	Substituting this result into Equation (\ref{rewrite_widetilde_varphi}), we get
	\begin{equation*}
		\begin{aligned}
			\varphi_n^{\prime} &=\dfrac{3}{n+1}\sum_{i=1}^n\left( \dfrac{2}{3}-|U_i-V_i|-U_i(1-U_i)-V_i(1-V_i)\right)+O_p(\dfrac{1}{n}) \\
			&=\varphi_n^{\prime\prime}+O_p(\dfrac{1}{n}) .
		\end{aligned}
	\end{equation*}
	Additionally, utilizing the results from Lemma \ref{lemmaA1}, it is easy to derive that
	\begin{eqnarray*}
		\mathrm{E}\varphi_n^{\prime\prime}=0,
	\end{eqnarray*}	
	\begin{eqnarray*}
		\mathrm{Var}(\varphi_n^{\prime\prime})&=&\left( \dfrac{3}{n+1}\right)^2\times n\mathrm{Var}\left( \dfrac{2}{3}-|U_i-V_i|-U_i(1-U_i)-V_i(1-V_i)\right) \\
		&=&n\left( \dfrac{3}{n+1}\right)^2\times \left[\mathrm{Var}(|U_1-V_1|)+\mathrm{Var}(U_1(1-U_1))\times 2+ 2\mathrm{Cov}(|U_1 -V_1 |,U_1 (1 -U_1 )) \right] \\
		&=&n\left( \dfrac{3}{n+1}\right)^2\times \left( \dfrac{1}{18}+\dfrac{1}{180}\times 2+2\times\left( -\dfrac{1}{180}\right) \times 2 \right) \\
		&=&\dfrac{2n}{5(n+1)^2}.
	\end{eqnarray*}
	Thus, this proof is complete.
\end{proof}

\begin{proof}[\textbf{Proof of Theorem \ref{theorem2.3}}]
	$\varphi_n^{\prime\prime}$ can be expressed as the sum of independently and identically distributed random variables with existing second moment, so its asymptotic normality can be easily obtained through the ordinary central limit theorem. By further utilizing the asymptotic representations of Theorem \ref{representation1} and Theorem \ref{representation2}, the asymptotic normality of $\varphi_n^{\prime}$ and $\varphi_n$ is also apparent.
\end{proof}

\newpage

\bibliographystyle{elsarticle-harv}
\bibliography{Representations}
\end{document}